\documentclass[10pt]{amsart}
\usepackage{amssymb}
\usepackage{amsmath}
\usepackage{epsfig}
\usepackage{graphics}

\font\tenmath=msbm10 \font\sevenmath=msbm7 \font\fivemath=msbm5
\newfam\mathfam \textfont\mathfam=\tenmath
\scriptfont\mathfam=\sevenmath \scriptscriptfont\mathfam=\fivemath

\def \\ { \cr }

\newcommand{\RR}{{\mathbb R}}

\newcommand{\NN}{{\mathbb N}}
\newcommand{\ZZ}{\mathbb Z}
\newcommand{\QQ}{{\mathbb Q}}

\def\by{{ \bf y}}
\def\bx{{ \bf x}}
\def\bz{{ \bf z}}
\def\bw{{ \bf w}}
\def\bv{{ \bf v}}
\def\bu{{ \bf u}}

\def\ba{{ \bf a}}
\def\bb{{ \bf b}}

\def\M{{\mathcal M}}

\def\S{{\mathcal S}}

\newcommand{\IET}{T_{(\lambda,\pi)}}

\newcommand{\AIET}{f}
\newcommand{\xs}{X_{\sigma}}

\numberwithin{equation}{section}

\newtheorem{theo}{Theorem}
\newtheorem{prop}[theo]{Proposition}

\newtheorem{lemma}[theo]{Lemma}

\theoremstyle{remark}

\newcommand{\SFTU}{\overline{\S}}
\newcommand{\SFTD}{\underline{\S}}

\begin{document}

\parindent = 0cm

\title
{Persistence of wandering intervals in self-similar affine interval exchange transformations }

\author{Xavier Bressaud}
\address{Institut de Math\'ematiques de Luminy, 163 avenue de Luminy, Case
907, 13288 Marseille Cedex 9, France.}
\email{bressaud@iml.univ-mrs.fr}

\author{Pascal Hubert}
\address{Laboratoire Analyse, Topologie et Probabilit\'es, 
Case cour A,
Facult\'e des Sciences de  Saint-Jer\^ome,
Avenue Escadrille Normandie-Niemen,
13397 Marseille Cedex 20, France.} \email{hubert@cmi.univ-mrs.fr}

\author{Alejandro Maass}
\address{Centro de Modelamiento Matem\'atico and Departamento de Ingenier\'{\i}a
Matem\'atica, Universidad de Chile, Av. Blanco Encalada 2120,
Santiago, Chile.} \email{amaass@dim.uchile.cl}

\subjclass{Primary: 37C15; Secondary: 37B10} \keywords{interval exchange transformations, substitutive systems, wandering sets}

\begin{abstract}
In this article we prove that given a self-similar  interval exchange transformation
$T_{(\lambda, \pi)}$, whose associated matrix verifies a quite general algebraic condition, 
there exists an affine interval exchange transformation  with wandering intervals that 
is semi-conjugated to it. That is, in this context the existence of Denjoy counterexamples occurs very often, generalizing the result of  M. Cobo in  \cite{Co}. 
\end{abstract}

\date{July 7, 2007}

\maketitle \markboth{Xavier Bressaud, Pascal Hubert, Alejandro
Maass}{Persistence of wandering intervals in self-similar AIET}

%%%%%%%%%%%%%%%%%%%%%%%%%%%%%%%%%%%%%%%%%%
 \section{Introduction}
%%%%%%%%%%%%%%%%%%%%%%%%%%%%%%%%%%%%%%%%%%

Since the work of Denjoy \cite{De}  it is known that every $C^1$-diffeomorphism of
the circle such that the logarithm of its derivative is a function of
bounded variation has no wandering intervals. There is no analogous result
for interval exchange transformations. Levitt in \cite{Le} found an example of a
non-uniquely ergodic affine interval exchange transformation with
wandering intervals. Latter, Camelier and Gutierrez \cite{Ca}, using Rauzy
induction technique exhibited a uniquely ergodic affine interval
exchange transformation with wandering intervals. Moreover, this example
is semi-conjugated to a self-similar interval exchange transformation. In
geometric language, it means that this inter- val exchange transformation
is induced by a pseudo-Anosov diffeomorphism. In combinatorial terms, the
symbolic system is generated by a substitution

An interval exchange  transformation (IET) is defined by the length of the intervals $\lambda = (\lambda_1, \ldots, \lambda_r)$ and a permutation $\pi$. It is denoted by $T_{(\lambda, \pi)}$.
To define an affine interval exchange transformation (AIET) one additional 
information is needed; the slope of the map on each interval. This is a vector $(w_1, \ldots, w_r)$ with $w_i >0$ for $i = 1, \ldots, r$. Camelier and Gutierrez remarked that a necessary condition for an AIET to be conjugated to the interval exchange transformation  $T_{(\lambda, \pi)}$ is that the vector $\log(w) = (\log(w_1), \dots, \log(w_r))$ is orthogonal to $\lambda$.

The conjugacy of an affine interval exchange transformation with an interval exchange transformation was studied in details by Cobo \cite{Co}. He proved that the regularity of the conjugacy depends on the position of the vector $\log(w)$ in the flag of the lyapunov exponents of the Rauzy-Veech-Zorich induction. 
In particular, assume that $T_{(\lambda, \pi)}$ is self-similar,  which means that $\lambda$ is an eigenvector of a positive $r\times r$ matrix $R$ obtained by applying Rauzy induction a finite number of times. Cobo proves that if $\log(w)$ belongs to the contracting space of $^tR$ then $f$ is $C^1$ conjugated to $T_{(\lambda, \pi)}$. If $\log(w)$ is orthogonal to $\lambda$ and is not in the contracting space of  $^tR$ then any conjugacy between $f$ and $T_{(\lambda, \pi)}$ is not an absolutely continuous function. Moreover, 
Camelier and Guttierez example shows that conjugacy between $f$ and $T_{(\lambda, \pi)}$ does not always exist. 

In this paper, we prove the following result:

\begin{theo} \label{main}
Let $T_{(\lambda, \pi)}$ be a self-similar  interval exchange transformation and $R$ the associated matrix obtained by Rauzy induction.  Let $\theta_1$ be the Perron-Frobenius eigenvalue of $R$. Assume that $R$ has an eigenvalue $\theta_2$ such that
\begin{enumerate}
\item \label{algebraic-hyp} $\theta_2$ is a conjugate of $\theta_1$,
\item \label{real-hyp}$\theta_2$ is a real number,
\item \label{sup1-hyp}$1 < \theta_2 (< \theta_1)$.
\end{enumerate}
 Then there exists an affine interval exchange transformation $f$ with wandering intervals that is semi-conjugated to $T_{(\lambda, \pi)}$.
\end{theo}

This result means that  Denjoy counterexamples occur very often (see section
\ref{pseudoAnosov}).

\subsection{Reader's guide}

Camelier-Gutierrez \cite{Ca} and Cobo \cite{Co} developed an strategy to prove the existence of a wandering interval in an affine interval exchange transformation $f$ which is semi-conjugated with a given IET. We explain it in section \ref{proof-main}. This strategy allowed them to achieve a first concrete example. Here we explore the limits of this method in order to consider a large (and in some sense abstract) family of IET.  
Let $\IET$ be a self-similar interval exchange transformation with associated matrix 
$R$. Let $\gamma = (\gamma_1, \ldots, \gamma_r)$ be the vector of the logarithm of the slopes of the affine interval exchange transformation $f$. If $f$ admits a wandering interval $I$, the length  
$\vert f(I) \vert$ is equal to   $e^{\gamma_j} \vert I \vert$ if $I$ is contained in interval $j$. Roughly speaking, to create a wandering interval from the interval exchange transformation $\IET$, one blows up an orbit of $\IET$. The difficulty is to insure that the total length remains finite. More precisely, if the symbolic coding of the orbit is $x =(x_n)_{n \in \ZZ}$, we have to check that the series
\begin{equation} \label{finitesums}
\sum_{n\geq 1} e^{{-\gamma(x_0) - \ldots -  \gamma(x_{n-1})}}  \text{ and } \sum_{n\geq 1} 
e^{{\gamma(x_{-n}) +\ldots + \gamma(x_{-1})}}
\end{equation} 
converge. This is certainly not true for a generic point $x$ of the symbolic system associated to $\IET$. Let $\ell(x)$ be the broken line with vertices $(n, \gamma(x_0) + \ldots +\gamma(x_{n-1}))_{n\in \NN}$ and $(n, \gamma(x_{-n}) +\ldots + \gamma(x_{-1}))_{n\geq 1}$. Since $\gamma$ is orthogonal to $\lambda$, for a generic point $x$, the line $\ell(x)$ oscillates around 0 as predicted by H\'alasz's Theorem (\cite{Ha}). If the vector $\gamma$ is not in the contracting space of $^tR$ the amplitude of the oscillations tends to infinity with speed  $$n^{\log(\theta_1)/\log(\theta_2)}.$$
It is hoped that the series \eqref{finitesums} converge if the $y$-coordinate of the broken line $\ell(x)$ is always positive and tends  to infinity fast enough as $n$ tends to $\pm \infty$. Points with this property are called {\it minimal points}. Those are the main tool of the paper. 

This analysis applies to a very large class of substitutions and not only to substitutions arising from interval exchange transformations. Section \ref{minimalpoints} gives an algorithm to construct minimal points. We prove that the prefix-suffix decomposition of any minimal point is ultimately periodic. From this analysis, we deduce that for any minimal point $x$ one has
\begin{equation} \label{speed}
\liminf_{n\to \infty} \frac{\gamma(x_0) +\ldots + \gamma(x_n)}
{ n^{\frac{\log(\theta_2)}{\log (\theta_1)}}} >0
\text{       and      }  
\liminf_{n\to \infty} \frac{-\gamma(x_{-n})-\ldots - \gamma(x_{-1})}{ n^{\frac{\log(\theta_2)}{\log (\theta_1)}}} >0
\end{equation}
Formulas in \eqref{speed} imply immediately the convergence of the series in \eqref{finitesums}. Moreover, formulas in \eqref{speed} has its own interest. It is a strengthening of a result by Adamczewski \cite{Ad} about discrepancy of substitutive systems.

Even if the fractal curves studied by Dumont and Thomas in
\cite{DT1}, \cite{DT2} are not considered explicitly in the article, they were a source of inspiration for the authors. These curves correspond to the renormalization of the broken lines $\ell(x)$ and appear in subsection \ref{series} in another language.

In section \ref{pseudoAnosov}, we discuss the hypothesis of the main result in a geometric language. We exhibit many examples where our hypothesis on the matrix $R$ are fulfilled.

%%%%%%%%%%%%%%%%%%%%%%%%%%%%%%%%%%%%%%%%%%%%%%%%%
\section{Preliminaries}
%%%%%%%%%%%%%%%%%%%%%%%%%%%%%%%%%%%%%%%%%%%%%%%%%

\subsection{Words and sequences}

Let $A$ be a finite set. One calls it an {\it alphabet} and its elements {\it symbols}. 
A {\it word} is a finite sequence of symbols in $A$, $w=w_0\ldots w_{\ell-1}$. The length of $w$ is denoted $|w|=\ell$. One also defines the empty word $\varepsilon$.  The set of words in the alphabet $A$ is denoted $A^*$ and $A^+=A^*\setminus \{\varepsilon\}$.  We will need to consider words indexed 
by integer numbers, that is, $w=w_{-m} \ldots w_{-1}.w_0\ldots w_{\ell}$ where $\ell,m \in \NN$ and the dot separates negative and non-negative coordinates. If necessary we call them {\it dotted  words}.  

The set of one-sided infinite sequences $x=(x_i)_{i \in \NN}$ in $A$ is denoted by $A^\NN$. Analogously, $A^\ZZ$ is the set of two-sided infinite sequences 
$x=(x_i)_{i\in \ZZ}$. 

Given a sequence $x$ in $A^+$, $A^\NN$ or $A^\ZZ$ one denotes $x[i,j]$ the sub-word of $x$ appearing between indexes $i$ and $j$. Similarly one defines $x(-\infty,i]$ and $x[i,\infty)$. 
Let $w=w_{-m} \ldots w_{-1}.w_0\ldots w_{\ell}$ be a (dotted) word in $A$. One defines the cylinder 
set $[w]$ as $\{ x \in A^\ZZ : x[-m,\ell]=w \}$. 

The shift map $T:A^\ZZ \to A^\ZZ$ or $T:A^\NN \to A^\NN$ is given by $T(x)=(x_{i+1})_{i \in \NN}$ for 
$x=(x_i)_{i\in \NN}$. A subshift  is any shift invariant and closed (for the product topology) subset of $A^\ZZ$ or $A^\NN$. A subshift is minimal if all of its orbits by the shift are dense. 

In what follows we will use the shift map in several contexts, in particular restricted to a subshift. 
To simplify notations we keep the name $T$ all the time. 

\subsection{Substitutions and minimal points} \label{subst} We refer to \cite{Qu} and \cite{fogg} and references therein for the general theory of substitutions.

A {\it substitution} is a map  $\sigma: A \to A^+$. It naturally extends to $A^+$, $A^\NN$ and $A^\ZZ$; for
$x=(x_i)_{i\in \ZZ} \in A^{\ZZ }$ the extension is given by $$
\sigma(x)=\ldots \sigma(x_{-2})\sigma(x_{-1}). \sigma(x_0)\sigma(x_1) \ldots
$$
where the central dot separates negative and non-negative coordinates
of $x$. A further natural  convention is that  the image of the empty word
$\varepsilon$ is $\varepsilon$.

Let $M$ be the matrix with indices in $A$ such that $M_{ab}$ is the number of times letter $b$ appears in $\sigma(a)$ for any $a,b \in A$. The substitution is primitive if there is $N >0$ such that for any $a \in A$, $\sigma^N(a)$ contains any other letter of $A$ (here $\sigma^N$ means $N$ consecutive iterations of $\sigma$). Under primitivity one can assume without loss of generality that $M>0$.

Let $\xs \subseteq A^\ZZ$ be the subshift defined from $\sigma$. That is, 
$x \in \xs$ if and only if any subword of $x$ is a subword of $\sigma^N(a)$ for some $N \in \NN$ and $a \in A$. 
 
Assume $\sigma$ is primitive. Given a point $x \in \xs$ there exists a unique sequence $(p_i,c_i,s_i)_{i\in \NN} \in (A^* \times A \times A^*)^\NN$ such that for each $i \in \NN$: $\sigma(c_{i+1})=p_ic_is_i$ and
$$\ldots \sigma^3(p_3) \sigma^2(p_2)\sigma^1(p_1) p_0 . c_0 s_0 \sigma^1(s_1)\sigma^2(s_2)\sigma^3(s_3)\ldots$$
is the central part of $x$, where the dot separates negative and non-negative coordinates. This sequence is called the prefix-suffix decomposition of $x$ (see for instance \cite{CS}). 

If only finitely many suffixes $s_i$ are nonempty, then there exists $a \in A$ and non-negative integers $\ell$ and $q$ such that 
$$x[0,\infty)=c_0 s_0 \sigma^1(s_1)\ldots \sigma^\ell(s_\ell)  \lim_{n\to \infty} \sigma^{n q}(a)$$
Analogously,  if only finitely many $p_i$ are non empty, then 
$$x(-\infty,-1]=\lim_{n\to \infty} \sigma^{n p}(b) \sigma^m(p_m) \ldots \sigma^1(p_1) p_0$$
for some $b \in A$ and non-negative integers $p$ and $m$.

Let $\theta_1$ be the Perron-Frobenius eigenvalue of $M$. Let $\lambda=(\lambda(a): a \in A)^t$ be a strictly positive right eigenvector of $M$ associated to $\theta_1$.  We will also assume the following algebraic property that we call (AH): $M$ has an eigenvalue $\theta_2$ which is a conjugate of $\theta_1$. Notice that this property coincides with hypothesis
(\ref{algebraic-hyp}) of Theorem \ref{main}.

The following lemma are important consequences of the algebraic property (AH). 

\begin{lemma}\label{conjugate}
Let $\eta:\QQ[\theta_1] \to \QQ[\theta_2]$  be the field homomorphism that sends $\theta_1$ to $\theta_2$. The vector $\gamma = \eta(\lambda)=(\eta(\lambda(a)): a \in A)^t$ is an eigenvector of $M$ associated to $\theta_2$.
\end{lemma}

\begin{proof}
The field homomorphism $\eta$ naturally extends to $\QQ[\theta_1] ^{|A|}$.
Since $\lambda$ belongs to $\QQ[\theta_1]^{|A|}$ (up to normalization), then one deduces that 
$M\eta(\lambda)=\theta_2 \eta(\lambda)$.  Thus, $\eta(\lambda)$ is an eigenvector of $M$ associated to $\theta_2$.
\end{proof}

\begin{lemma}\label{lem:alg}
Let $\gamma$ be the eigenvector of $M$ associated to $\theta_2$ as in Lemma \ref{conjugate}. Then for any $|A|$-tuple of non-negative integers $(n_a: a \in A)$,  $\sum_{a \in A} n_a \gamma(a) = 0$
implies  $n_a= 0$ for any $a \in A$.
\end{lemma}

\begin{proof}
Assume $\sum_{a \in A} n_a \gamma(a)  = 0$. Since $\gamma=\eta(\lambda)$, applying 
$\eta^{-1}$ one gets that 
 $\sum_{a \in A} n_a \lambda(a) = 0$. This equality implies that $n_a = 0$ for every $a \in A$
because the coordinates of $\lambda$ are positive.
\end{proof}

Let  $\gamma=\eta(\lambda)$ as in Lemma \ref{conjugate}.  For 
$w=w_0 \ldots w_{l-1} \in A^+$ denote $\gamma(w)=\gamma (w_0)+\ldots+\gamma(w_{l-1})$. 

Let $x \in \xs$. Define $\gamma_0(x)=0$, $\gamma_ n(x)=\sum_{i=0}^{n-1}
\gamma(x_i)$ for $n >0$ and $\gamma_ n(x)=\sum_{i=n}^{-1} \gamma({x_i})$ for $n <0$. 
Put $\Gamma(x)=\{\gamma_n(x): n\in \ZZ\}$.
In a similar way, given a (dotted) word $w=w_{-m}\ldots  w_0\ldots w_{l-1}$ one defines 
$\gamma_0(w)=0$, $\gamma_ n(w)=\sum_{i=0}^{n-1} \gamma(w_i)$ for $0<n \leq l$, 
$\gamma_ n(w)=\sum_{i=n}^{-1} \gamma(w_i)$ for $-m \leq n <0$ and the set $\Gamma(w)$. 

The best occurrence of a symbol $a \in A$ in $w$ is 
$-m \leq i < l$  such that $w_i=a$ and 
$\gamma_{i+1}(w)=\min\{ \gamma_{j+1}(w): -m \leq j < l, w_j=a \}$.
By Lemma \ref{lem:alg}, under hypotheses (AH) this number is well defined and unique. 

One says $x$ is {\it minimal} if $\gamma_n(x) \geq  0$ for any $n \in \ZZ$. The set of minimal points for 
$\sigma$ is denoted by $\M(\sigma)$. It is important to mention that if $x$ is a minimal point of a substitution satisfying hypothesis (AH)  then, by Lemma \ref{lem:alg},  
$\gamma_n(x)>0$ whenever $n \not = 0$.

\subsection{Affine interval exchange transformations}

Let $0=a_0< a_1 < \ldots < a_{r-1} < a_r=1$ and  $A=\{1,\ldots,r\}$.

An {\it affine interval exchange transformation} (AIET) is a bijective map $\AIET: [0,1) \to [0,1)$  of the form $f(t)= w_i t + v_i$ if $t \in [a_{i-1},a_i)$ for $ i\in A$. The vector $w=(w_1,\ldots,w_r)$ 
is called the slope of $\AIET$. We assume furthermore the slope is strictly positive. 

An {\it interval exchange transformation} (IET) is an AIET with slope $w=(1,\ldots,1)$. Commonly an IET is given by a vector  $\lambda=(\lambda_1,\ldots, \lambda_r)$ such that 
$\lambda_i=|a_i-a_{i-1}|$ for $i \in A$ and a permutation $\pi$ of  $A$ which indicates the way intervals 
$[a_{i-1},a_i)$'s are rearranged by the IET. Clearly,  $a_i=\sum_{j=1}^i \lambda_j$. We use $\IET$ to refer to the IET associated to $\lambda$ and $\pi$.

One says the AIET $\AIET$ is semi-conjugated with the IET $\IET$ if there is a monotonic, surjective and continuous  map $h: [0,1)\to [0,1)$ such that $h\circ \AIET=\IET \circ h$.

Let $\IET$ be an interval exchange transformation. 
There is a natural symbolic coding of the orbit of any point $t \in [0,1)$  by $\IET$.  Consider the partition 
$\alpha=\{[0,a_1),\ldots,[a_{i-1},a_i),\ldots,[a_{r-1},1)\}$ and define 
$\phi(t)=(x_i)_{i\in \ZZ} \in A^\ZZ$ by $t_i=j$ if and only if $\IET^i(t) \in [a_{j-1},a_j)$. The set 
$\phi([0,1))$ is invariant for the shift but it is not necessarily closed, then one considers its closure 
$X=\overline{\phi([0,1))}$. This procedure produces  a semi-conjugacy (factor map) $\varphi: (X,T)\to 
([0,1),\IET)$. If $t$ is not in the orbit of the extreme points $0,a_1,\ldots,1$, then it has a unique preimage by $\varphi$. If not, it has at most two preimages corresponding to the coding  of $(\IET^i(\lim_{s\to t^-} s))_{i\in \ZZ}$.

We use freely concepts related to Rauzy-Zorich-Veech induction. Rauzy induction was defined  in  \cite{Ra}, extended to zippered rectangles by  Veech \cite{Ve}, and accelerated by  Zorich \cite{Zo}. For a complete description about the Rauzy-Veech-Zorich induction see also  the expository papers by Zorich \cite{Zo2} and Yoccoz \cite{Yo}.    

An IET  $\IET$ is {\it self-similar} if it can be recovered from itself after finitely many steps of Rauzy inductions (up to normalization). More precisely,  there exists a loop in the Rauzy diagram and an associated {\it Perron-Frobenius} matrix $R$ such that 
$$\theta_1\lambda = R \lambda$$
\noindent with $\theta_1$ the dominant eigenvalue of $R$. 
 
For a  self-similar IET $\IET$ there is a direct relation between the subshift $X$ and the matrix $R$ associated to  $\IET$. Indeed, there exists a substitution $\sigma: A \to A^+$ with
associated matrix $M= {^tR}$ such that $\xs=X$ (see \cite{Ca} and references therein). If the IET $\IET$ is minimal then the subshift $\xs$ is minimal too. In the sequel, we will use the fact that  the substitution $\sigma$ is primitive which implies that $\xs$ is minimal. Nevertheless, no specific property of substitutions obtained from $\IET$ will be needed for our purpose. 

The relation between self-similar IET and pseudo-Anosov diffeomorphisms is explained in \cite{Ve}. 

%%%%%%%%%%%%%%%%%%%%%%%%%%%%%%%%%%%%%%%%%%%%%%%%%
\section{Construction of minimal points} \label{minimalpoints}
%%%%%%%%%%%%%%%%%%%%%%%%%%%%%%%%%%%%%%%%%%%%%%%%%

Let $\sigma: A \to A^+ $ be a primitive substitution with associated matrix  $M > 0$. 
Let $\theta_1$, $\theta_2$, $\lambda$ and $\gamma$ be as in subsection \ref{subst}. In addition, assume $\theta_2$ verifies the hypotheses of Theorem \ref{main}. By Perron-Frobenius theorem, $\gamma$ has  negative and positive coordinates. The main objective of the section is to give a combinatorial construction of minimal points in this case. 

\vfill\eject

\subsection{Existence of minimal points} 

\begin{lemma}\label{lem:growth}
Let $a \in A$ such that $\gamma(a)>0$ and $n\in \NN$. 
Write $\sigma^n(a)=p_ns_n$ where the minimum of $\Gamma(\sigma^n(a))$  is 
attained at $\gamma_i(\sigma^n(a))$ and $i=|p_n|$. Then $\gamma(s_n) \geq \theta_2^n \gamma(a)$. In particular $|s_n|$ grows exponentially fast with $n$. 
\end{lemma}
\begin{proof}
Observe that $\gamma(p_n)+\gamma(s_n)= \theta_2^n \gamma(a)$ and 
$\gamma(p_n) \leq 0$.
\end{proof} 

\begin{lemma}
$\M(\sigma) \not = \emptyset$
\end{lemma}
\begin{proof}
Since $\gamma$ has positive and negative coordinates and $\xs$ is minimal,  then there exist  $b,c \in A$ such that  $bc$ is a subword  of a point in $\xs$ and $\gamma(b)<0, \gamma(c)>0$ holds. 

Let $n \geq 0$ and define $u_n=\sigma^n(b).\sigma^n(c)$. The sequence $\Gamma(u_n)$ attains its minimum at some $N_n \in \{-|\sigma^n(b)|,\ldots,-1,0,\ldots,|\sigma^n(c)|\}$. Define the (dotted)word  
$v_n=u_n[-|\sigma^n(b)|,N_n-1].u_n[N_n,|\sigma^n(c)|-1]=v_n^-.v_n^+$.  The minimum of $\Gamma(v_n)$ is attained at coordinate $0$, and is equal to $0$. 

By Lemma \ref{lem:growth} there is a subsequence $(n_i)_{i\in \NN}$ such that $$\lim_{i \to \infty} |v_{n_i}^-|=\lim_{i \to \infty} |v_{n_i}^+|=\infty$$ By compactness and eventually taking once again a  subsequence there exists $x\in X$ such that
for any $m \in \NN$  there is $i\in \NN$ with $n_i \geq m$ and $x \in [v_{n_i}^-.v_{n_i}^+]$. 
Thus $\Gamma(x[-m,m]) \subseteq \RR^+$ and its minimum is zero at zero coordinate. This implies $x \in \M(\sigma)$. 
\end{proof}

\subsection{The {\it best strategy} algorithm} 
In what follows we develop a procedure to construct {\it minimal} points that will become useful in next subsections. 

The following two lemma follow directly from equality $M\gamma=\theta_2 \gamma$. Their simple proofs are left to the reader. 

\begin{lemma}\label{lem:profile}
Let $m \in \NN$ and $w  \in A^+$. Then 
$\gamma(\sigma^{m}(w))=\theta_2^m \gamma(w)$. 
\end{lemma}

\begin{lemma}\label{lem:fractalgamma}
Let $w=w_0\ldots w_{l-1} \in A^+$. Write $\sigma(w)=\sigma(w_0)\ldots \sigma(w_{l-1})$. 
The minimum of $\Gamma(\sigma(w))$  is attained in a coordinate corresponding  to some $\sigma(w_i)$, where 
$w_i$ is the best occurrence of this symbol in $w$. 
\end{lemma}

%\begin{figure}
%\caption{This figure illustrates the proof of Lemma  \ref{lem:fractalgamma}. 
%\label{fig:lemmas}}
%\end{figure}

\subsubsection{The basic procedure} The following procedure will allow to construct 
the prefix-suffix decomposition of a minimal point. 

{\bf Step 0:} For each $a \in A$ write $\sigma(a)=p_0^{a,0} c_0^{a,0} s_0^{a,0}$ where 
$\Gamma(\sigma(a))$ attains its minimum at $\gamma_{|p_0(a)|}(\sigma(a))$.

{\bf Step 1:} Let $a \in A$. By Lemma \ref{lem:fractalgamma}, the minimum of 
$\Gamma(\sigma^2(a))$ comes from $\sigma(b)$ for some $b\in A$ in its best occurrence in $\sigma(a)$.  Write $\sigma(a)=p_1^{a,1} c_1^{a,1} s_1^{a,1}$ where $c_1^{a,1}=b$ is the best occurrence of $b$ in $\sigma(a)$. Put 
$w_1(a)= \sigma(p_1^{a,1}) p_0^{b,0} . c_0^{b,0} s_0^{b,0} \sigma(s_1^{a,1})$, where the dot  separates negative and non-negative coordinates. Let $p_0^{a,1}=p_0^{b,0}$, $c_0^{a,1}=c_0^{b,0}$ and 
$s_0^{a,1}=s_0^{b,0}$. The sequence $(p_i^{a,1},c_i^{a,1},s_i^{a,1})_{i=0}^ 1$ is called the best strategy for symbol $a$ at step 1. By construction $\Gamma(w_1(a))\subseteq \RR^+$ and the minimum is equal to zero at coordinate zero.

{\bf Step n+1:} assume in previous step we have constructed for each symbol $a \in A$ the best strategy 
$(p^{a,n}_i,c^{a,n}_i,s^{a,n}_i)_{i=0}^ n$. This sequence verifies: 

(i) for $0 \leq i \leq n$, $\sigma(c_{i+1}^{a,n})=p_{i}^{a,n} c_{i}^{a,n} s_{i}^{a,n}$ 
(here $c_{n+1}^{a,n}=a$). Moreover, each $c_i^{a,n}$ is the best occurrence of this symbol in 
$\sigma(c_{i+1}^{a,n})$. 

(ii) $\Gamma(w_n(a)) \subseteq \RR^+$ and its minimum is zero at zero coordinate, where 
$$w_n(a)=\sigma^{n+1}(a)=\sigma^n(p_{n}^{a,n})\ldots \sigma(p_1^{a,n})  p_0^{a,n}. c_0^{a,n}  s_0^{a,n} 
\sigma(s_{1}^{a,n})\ldots \sigma^n(s_n^{a,n})$$

Now we proceed as in step 1.  Consider $a \in A$.  By Lemma \ref{lem:fractalgamma}, 
the minimum of $\Gamma(\sigma^{n+2}(a))$ comes from $\sigma^{n+1}(b)$ for some $b\in A$ in its best occurrence in $\sigma(a)$. Write $\sigma(a)=p_{n+1}^{a,n+1} c_{n+1}^{a,n+1} s_{n+1}^{a,n+1}$ where 
$c_{n+1}^{a,n+1}=b$ is the best occurrence of $b$ in $\sigma(a)$. The finite sequence 
$(p^{a,n+1}_i,c^{a,n+1}_i,s^{a,n+1}_i)_{i=0} ^{n+1}$ where 
$(p^{a,n+1}_i,c^{a,n+1}_i,s^{a,n+1}_i)=$ \break $(p^{b,n}_i,c^{b,n}_i,s^{b,n}_i)$ for $0\leq i \leq n$ is a best strategy for $a$ at step $n+1$ and 
verifies conditions (i) and (ii) by construction.

\subsubsection{Finitely many  minimal points.}
For each $a \in A$ and $n \in \NN$ consider the cylinder set $C^{a,n}=[w_n(a)]$, where $w_n(a)$ is the dotted word defined in previous subsection. It is clear from the basic procedure that for any $a \in A$ and $n\in \NN$ there exists a unique $b \in A$ such that $C^{a,n+1}\subseteq C^{b,n}$. Thus, by compactness, there exist at most $|A|$ infinite decreasing sequences of the form $(C^{a_n,n})_{n \in \NN}$. Let $C_1,\ldots,C_{\ell}$ with $\ell \leq |A|$ be the collection of intersections of such sequences.  
 Remark that such sets are finite. 

Given a minimal point $x \in X$ with prefix-suffix decomposition 
$(p_i,c_i,s_i)_{i\in \NN}$ and $n \in \NN$, there is 
$a_n \in A$ such that $(p_i,c_i,s_i)=(p_i^{a_n,n},c_i^{a_n,n},s_i^{a_n,n})$ for $0\leq i \leq n$. 
Therefore,  $x \in C_i=\bigcap_{n \in \NN} C^{a_n,n}$ for some $1\leq i \leq \ell$. 

The following proposition is plain.
\begin{prop}\label{prop:finitenumber}
There are finitely many minimal points.
\end{prop}

We will see later that minimal points have ultimately periodic prefix-suffixe decom- 
position. This fact yields to an alternative proof of previous proposition. 

\subsection{Serie associated to a minimal point}\label{series}

Define $\SFTU=\{(p_i,c_i,s_i)_{i\in \NN} : \forall \ i >0, \  \sigma(c_i)=p_{i-1}c_{i-1}s_{i-1} \}$ and 
$\SFTD=\{(p_i,c_i,s_i)_{i\in \NN} : \forall \ i \geq 0, \  \sigma(c_i)=p_{i+1}c_{i+1}s_{i+1} \}$. Observe that finite sequences taken from sequences in $\SFTU$ and $\SFTD$ coincide once reversed. 

Let $a \in A$ and $n \geq 1$. Then $\sigma^n(a)$ can be decomposed as 
$$\sigma^n(a)=\sigma^{n-1}(p_1) \ldots \sigma(p_{n-1})  p_n c_n s_n \sigma(s_{n-1}) \ldots \sigma^{n-1}(s_1) $$
where for all $1 \leq i \leq n$, $\sigma(c_{i-1})=p_ic_is_i$ (we have considered $c_{0}=a$). This decomposition is not unique. To $a$ and the finite sequence $(p_i,c_i,s_i)_{i=1}^n$ one associates the finite sum:
$$v(a; (p_i,c_i,s_i)_{i=1}^n)=\sum_{i=1}^n \theta_2^{-i} \gamma(p_i)$$
Clearly, 
given ${\bx}=(p^\bx_i,c^\bx_i,s^\bx_i)_{i \in \NN} \in \SFTD$ with $c^\bx_0=a$, the series 
$$v(a;{\bx})=\lim_{n\to \infty} v(a;(p^\bx_i,c^\bx_i,s^\bx_i)_{i=1}^n)=\sum_{i\geq 1} \theta_2^{-i} \gamma(p^\bx_i)$$ exists.

Let $v(a)=\min\{v(a;{\bx}): {\bx} \in \SFTD \text{ with } c^\bx_0=a\}$. A sequence  ${\bx} \in \SFTD$ with $c^\bx_0=a$ such that $v(a;{\bx})=v(a)$ is said to be {\it minimal for $a$}.

The best strategy for symbol $a$ at step $n \geq 1$ given by the  algorithm produces a finite sequence $(p^{a,n}_i,c^{a,n}_i,s^{a,n}_i)_{i=0}^n$.  Set 
$v_n(a)=\sum_{i=0}^n \theta^{-n+i-1}  \gamma(p_i^{a,n})$. It follows that 
$v_n(a)=v(a; (p^{a,n}_{n-i},c^{a,n}_{n-i},s^{a,n}_{n-i})_{i=0}^n)$. 

\begin{lemma}\label{lem:v(a)exists}
For every $a \in A$ and $n \geq 1$,  $v_n(a)$ is minimal among the \break $v(a; (p_i,c_i,s_i)_{i=1}^{n+1})$ and 
$v(a)=\lim_{n\to\infty}v_n(a)$.
\end{lemma}
\begin{proof}
The first fact is analogous to say that $(p^{a,n}_i,c^{a,n}_i,s^{a,n}_i)_{i=0}^n$ is the best strategy.
Moreover,   
$|v_n(a) - v(a)| \leq K \theta_2^{-n}$ for some constant $K>0$. This implies the desired result.  
\end{proof}

\begin{lemma}
\label{lem:series}
Let $a \in A$. Assume there is a finite sequence  $(\bar p_j, \bar c_j, \bar s_j)_{j=1}^l$  such that 
for infinitely many $n \in \NN$,  
$(p^{a,n}_{n-j+1},c^{a,n}_{n-j+1},s^{a,n}_{n-j+1})_{j=1}^l=(\bar p_j, \bar c_j, \bar s_j)_{j=1}^l$.
Then, there exists $\by=(p^\by_i,c^\by_i,s^\by_i)_{i\in \NN} \in \SFTD$ 
such that $(\by_j)_{j=1}^l=(\bar p_j, \bar c_j, \bar s_j)_{j=1}^l$, $c^\by_0=a$ and 
$v(a)=v(c_0^\by;\by)$.
\end{lemma} 
\begin{proof}
For any  $n \in \NN$ where the property of the lemma holds consider the point 
$$\by^{(n)}=\by^{(n)}_0\ldots \by^{(n)}_{n}=(p,a,s)(p^{a,n}_{n},c^{a,n}_{n},s^{a,n}_{n})\ldots(p^{a,n}_0,c^{a,n}_0,s^{a,n}_0)$$
where $\sigma(b)=pas$ for some $b \in A$.

Let $\by=(p^\by_i,c^\by_i,s^\by_i)_{i\in \NN}$ be the limit of a subsequence  $(\by^{(n_i)})_{i\in \NN}$. 
It follows by construction that $(\by_j)_{j=1}^l=(\bar p_j, \bar c_j, \bar s_j)_{j=1}^l$, $c_0^\by=a$ and
 $\sigma(c^{\by}_i)=p^{\by}_{i+1}c^{\by}_{i+1}s^{\by}_{i+1}$ for any $i \geq 0$.
 Also, $c^\by_{i+1}$ is the best occurrence of this symbol in 
$\sigma( c^\by_{i})$.

Let $\epsilon>0$ and $i_0\in \NN$ such that $|v(a)-v_{n_i}(a)|\leq \epsilon/2$  for $i\geq i_0$. 
Let $L \in \NN$ be such that $\theta_2^{-L}\leq \epsilon/4C$ where $C>0$ is such that 
$|\gamma(p^{c,n}_i)|/(\theta_2-1)\leq C$ for any  $c\in A$ and $n\in \NN$. Thus for $i$ enough large, 
$(p^\by_j,c^\by_j,s^\by_j)=(p_{n_i-j+1},c_{n_i-j+1},s_{n_i-j+1})$ for $0\leq j <L$ and 
$|v(a)-\sum_{i\geq 1} \theta_2^{-i}  \gamma(p^\by_i)| \leq \epsilon$. Since $\epsilon$ is arbitrary one concludes $v(c_0^\by)=v(a)=\sum_{i\geq 1} \theta_2^{-i} \gamma(p^\by_i)$.
\end{proof}

One says that a point $\by=(p_i^\by,c_i^\by,s_i^\by)_{i\in \NN} \in \SFTD$ verifies the {\it continuation property} if 
$v(c_i^\by)=v(c^\by_i;T^i(\by))$ for all $i\geq 0$, where $T$ is the shift map. 
It is clear that $T^i(\by)$ has the continuation property too, for any $i\in \NN$.  In fact to satisfy the continuation property it is enough to be minimal for $c_0^\by$. 

\begin{lemma}\label{contproperty}
If $\by=(p_i^\by,c_i^\by,s_i^\by)_{i\in \NN} \in \SFTD$ is minimal for $c_0^\by$ (that is, $v(c_0^\by)=v(c^\by_0;\by)$) then $\by$ verifies the continuation property.
\end{lemma}
\begin{proof}
Let $b=c_1^\by$ and $\bz=(p^\bz_i,c^\bz_i,s^\bz_i)_{i\in \NN} \in  \SFTD$ with $c_0^\bz = b$ and 
$v(b;\bz)=v(b)$ given by Lemma \ref{lem:series} (considering $l=0$). The sequence 
$\bw=\by_0 \by_1 T(\bz)$ belongs to $\SFTD$ and verifies $v(a;\bw)=\theta_2^{-1} \gamma(p_1^\by)+
\theta_2^{-1} v(b)$. Thus, if $v(b;T(\by)) > v(b)$, from 
$v(a)=v(a;\by)=\theta_2^{-1} \gamma(p_1^\by)+ \theta_2^{-1} v(b;T(\by))$, one deduces that 
$v(a;\bw)< v(a)$ which is a contradiction. 
\end{proof}

This lemma proves that sequences $\by$ constructed in Lemma \ref{lem:series} verifies  the continuation property.

\subsection{Minimal points are ultimately periodic}
In this section we prove that any minimal point $x \in \xs$ has ultimately  periodic prefix-suffix decomposition. That is, if  $\bar x=(p_i,c_i,s_i)_{ i \in \NN }$ is the prefix-suffix decomposition of $x$, 
then  $T^{p+q}({\bar x})=T^q{\bar x}$ for some $p>q \geq 0$. If $q=0$ one says $x$ is a periodic minimal point. 

\begin{lemma}
\label{lem:periodicmin}
For every $a \in A$ there exists a  ultimately periodic point \break
${\bx(a)}=(p^{\bx(a)}_i,c^{\bx(a)}_i,s^{\bx(a)}_i)_{i\in \NN}  \in \SFTD$ with $c^{\bx(a)}_0 = a$ and $v(a;{\bx(a)})=v(a)$ (so, $\bx(a)$ has the continuation property).

 \end{lemma}

\begin{proof}
Let $a \in A$ and $\by=(p^\by_i,c^\by_i,s^\by_i)_{i\in \NN} \in  \SFTD$ with $c_0^\by = a$ and 
$v(a;\by)=v(a)$ given by Lemma \ref{lem:series} (considering $l=0$). 
We are going to construct another one with ultimately periodic decomposition.  

Let $0 <  q < p$ be such that $\by_q=\by_p$ and $c_{q-1}^\by=c_{p-1}^\by=b$.  
The preperiodic sequence 
$\bx=\by_0 \ldots \by_{q-1} \by_{q}  \ldots \by_{p-1} \by_q \ldots \by_{p-1} \ldots \in \SFTD$  since 
$\sigma( c^\by_{p-1})=  p^\by_{q}  c^\by_{q} s^\by_{q}$ by hypothesis.  We are going to prove that 
$v(a;\bx)=v(a)$.  

Observe that, by Lemma  \ref{lem:series}, $$v(b)=\sum_{i\geq q} \theta_2^{-(i-q+1)} \gamma( p^\by_i) 
\text{ and } 
v(b)=\sum_{i\geq p} \theta_2^{-(i-p+1)} \gamma( p^\by_i) \ . $$ Thus, 
$v(b)= \sum_{i=q}^{p-1} \theta_2^{-(i-q+1)} \gamma( p^\by_i) + \sum_{i\geq p} \theta_2^{-(i-q+1)}  \gamma(p^\by_i)= \sum_{i=q}^{p-1} \theta_2^{-(i-q+1)} \gamma( p^\by_i) + \theta_2^{-(p-q)} v(b)$. If we denote 
$B= \sum_{i=q}^{p-1} \theta_2^{-(i-q+1)} \gamma( p^\by_i)$, then 
$v(b)= B  \sum_{i\geq 0} \theta_2^{-(p-q)i }$.  Consequently, 
$$v(a)= \sum_{i=1}^{q-1} \theta_2^{-i} \gamma( p^\by_i) + \theta_2^{-(q-1)} B \sum_{i\geq 0} \theta_2^{-(p-q)i }$$ 

On the other hand, a direct computation yields to 
$$v(a;\bx)=  \sum_{i=1}^{q-1} \theta_2^{-i} \gamma( p^\by_i) + \theta_2^{-(q-1)} ( \sum_{i\geq 0} \theta_2^{-(p-q)i } B ) \  ,$$

which implies, $v(a;\bx)=v(a)$. 
\end{proof}

To each preperiodic sequence $\bx(a)$ constructed in previous lemma one can associate a point
$x$ in the symbolic space $\xs$ with periodic prefix-suffix decomposition of period
$$(p_{0},c_{0},s_{0}),\ldots,  
(p^{}_{p-q},c^{}_{p-q},s^{}_{p-q})=
(p^{ \bx(a)}_{p-1},c^{ \bx(a)}_{p-1},s^{ \bx(a)}_{p-1}),\ldots,  
(p^{ \bx(a)}_{q},c^{ \bx(a)}_{q},s^{ \bx(a)}_{q}) \ . $$ 
Even if, by construction, this point is associated to the minimal value $v(b)$, 
there is no reason for it to be a minimal point. 

Without loss of generality we will do the following simplification. 
By iterating $\sigma$ enough times one can assume that all ultimately periodic sequences constructed in Lemma \ref{lem:periodicmin} are of period $1$ and of preperiod $1$. That is, 
for  each letter $a \in A$, $c_0^{\bx(a)}=a$ and $\bx_i=(p^{(a)},\hat a, s^{(a)})$ for all $i\geq 1$. 
The letter $a \in A$ is \emph{periodic} if $\hat a = a$ and one denotes  $\hat A$ the subset of periodic letters. Since,  the construction of Lemma \ref{lem:periodicmin} implies 
that $v(c_i^{\bx(a)})=v(a;T^i(\bx(a)))$ for $0\leq i \leq p-1$, then under this simplification $v(\hat a)=v(\hat a;T(\bx(a)))$.

\begin{lemma}\label{lem:replace}
Let $\by \in \SFTD$ verifying the continuation property. Then,  
for any $i\geq 1$ the point $\by^{(i)}=\by_0\ldots\by_{i} T(\bx(c_i^\by))$ has the continuation property 
too.
\end{lemma}
\begin{proof}
Let $i \geq 1$ and $1\leq j \leq i$. From the continuation property one deduces that 
$v(c^\by_j)=\sum_{k=1}^{i-j} \theta_2^{-k} \gamma(p^\by_{k+j})+\theta_2^{-(i-j)} v(c_i^\by)$. 
But, $v(c_i^\by)=v(c_i^\by;\bx(c_i^\by))$ and $v(\hat c_i^\by)=v(\hat c_i^\by; T(\bx (c_i^\by)))$, then 
$\by^{(i)}=\by_0\ldots\by_{i-1} T(\bx(c_i^\by))$ has the continuation property 
too.
\end{proof}

\begin{lemma}\label{lem:cola}
Let $\bx , \by \in \SFTD$ such that $(\bx_i)_{i\geq l+1}=(\by_i)_{i\geq l+1}$ and $c_0^\bx=c_0^\by=a$. 
If $v(a;\bx)= v(a;\by)$ then $(\bx_i)_{i\geq 1}=(\by_i)_{i\geq 1}$.
\end{lemma}
\begin{proof}
Let $\bx=(p_i^\bx,c_i^\bx,s_i^\bx)_{i\in \NN}$ and $\by=(p_i^\by,c_i^\bx,s_i^\by)_{i\in \NN}$.
From the hypothesis one deduces that  
$$\sum_{i=1}^l \theta_2^{-i} \gamma(p^{\bx}_i)= \sum_{i=1}^l \theta_2^{-i} \gamma(p^{\by}_i)$$
and consequently 
$$\gamma(\sigma^{l-1} (p^{\bx}_1)\ldots p^{\bx}_l)=\gamma(\sigma^{l-1}(p^{\by}_1)\ldots p^{\by}_l) \  .$$
But words 
$\sigma^{l-1} (p^{\bx}_1)\ldots p^{\bx}_l$ and $\sigma^{l-1}(p^{\by}_1)\ldots p^{\by}_l$
are prefixes of $\sigma^{l}(a)$. Then, by the algebraic condition (Lemma \ref{lem:alg}) they must be the same. This implies $(p^{\bx}_i,c^{\bx}_i, s^{\bx}_i)=(p_i^{\by},c_i^{\by},s_i^{\by})$ for $1\leq i \leq l$.  
\end{proof}

\begin{theo}\label{theo:periodicmin}
The prefix-suffix decomposition of any minimal point is ultimately periodic.
\end{theo}
\begin{proof}
Let $x \in \xs$ be minimal point with prefix-suffixe decomposition $(p_i,c_i,s_i)_{i\in \NN}$. 
There exists a finite sequence $(\bar p_j,\bar c_j, \bar s_j)_{j=0}^l$  such that 
$(\bar p_0,\bar c_0, \bar s_0)=(\bar p_l,\bar c_l, \bar s_l)$ and for infinitely many $i \in \NN$, 
$(p_{i-j},c_{i-j},s_{i-j})_{j=0}^l=(\bar p_j,\bar c_j, \bar s_j)_{j=0}^l$. 

Let $a=\bar c_0=\bar c_l$. By Lemma  ~\ref{lem:series}, there is 
a point $\by \in \SFTD$ verifying the continuation property
such that $(\by_j)_{j=0}^l=(\bar p_j,\bar c_j, \bar s_j)_{j=0}^l$.
In particular, $v(a;\by)=v(a)$ and $v(a;T^l(\by))=v(a)$. Since, 
$v(a)=v(a;\bx(a))$, then by Lemma \ref{lem:replace} the sequence 
$\bz= \by_0\ldots \by_{l} T(\bx(a))$ has the continuation property and  
$v(a)=v(a;\bz)$ holds. Therefore, by Lemma \ref{lem:cola}, one concludes that 
$(\bx(a))_{i\geq 1}= (\bz_i)_{i\geq 1}$. 

We have proved that $a \in \hat A$, that is $a = \hat a$, and that  the word 
$(p^{(a)}, a, s^{(a)})(p^{(a)},  a, s^{(a)})$  appears infinitely many times in the prefix-suffixe decomposition of $x$. Now we prove that $(p_i,c_i,s_i)_{i\in \NN}$ is ultimately periodic with period 
$(p^{(a)}, a, s^{(a)})$. 

Assume this result does not hold. Then there is $b\not = a$ in $A$ such that 
$$(p_i, c_i, s_i) (p_{i-1},c_{i-1},s_{i-1})(p_{i-2},c_{i-2},s_{i-2})=(p^{(a)}, a, s^{(a)}) (p^{(a)}, a, s^{(a)})(p,b,s)$$ 
for infinitely many $i \in \NN$. 

By Lemma ~\ref{lem:series}, there is a point $\bw \in \SFTD$ verifying the continuation property and such that $\bw_0\bw_1\bw_2=(p^{(a)}, a, s^{(a)}) (p^{(a)}, a, s^{(a)})(p,b,s)$. Since $v(b)=v(b;T^2(\bw))$ and $v(b)=v(b;\bx(b))$, by Lemma \ref{lem:replace}, 
the points $\bu=\bw_0 \bw_1 \bw_2 T(\bx(b))$ and $\bv=\bx(a)_0\bx(b)$  have the continuation property.
Since $\bu$ and $\bv$ are ultimately equal, then, by Lemma \ref{lem:cola}, one concludes 
$a=b$ which is a contradiction. This proves the theorem.
 \end{proof}

We stress the fact  that it is possible to construct examples with minimal points having ultimately periodic but not periodic prefix-suffix decomposition.

\subsection{Convergence of series associated to minimal points}

\begin{lemma}\label{lem:unico}
Let $\by \in \SFTD$ such that $c_0^\by=a \in \hat A$ and $v(a;\by)=v(a)$. Then, 
$\by_1=(p^{(a)},a,s^{(a)})$.
\end{lemma}
\begin{proof}
Put $c_0^\by=a$. First we prove that $v(c_1^\by)=v(c_1^\by;T(\by))$. 
Let $\bz= \by_0 \by_1 T(\bx(c_1^\by)) \in \SFTD$.  If the assertion is not true then 
$$v(a)=\theta_2^{-1} (\gamma(p^\by_1)+ v(c_1^\by;T(\by)))
> \theta_2^{-1} (\gamma(p^\by_1)+ v(c_1^\by) )= v(a;\bz) \geq v(a)$$
which is a contradiction.  Thus, $v(c_1^\by)=v(c_1^\by;T(\by))$ and  furthermore $v(a)=v(a;\bz)$.  

Then, the point $\bw=(p^{(a)},a,s^{(a)}) (p^{(a)},a,s^{(a)}) T(\bz)$ verifies $v(a)=v(a;\bw)$. But $\bw$ and 
$\bx(a)$ are ultimately equal, then by Lemma \ref{lem:cola}, $\by_1=(p^{(a)},a,s^{(a)})$.

\end{proof}

\begin{lemma}\label{lem:separation}
Let $x \in \xs$ be a minimal point. Then,  
$$\liminf_{n\to \infty} \frac{\gamma(x_0\ldots x_n)}{ n^{\frac{\log(\theta_2)}{\log (\theta_1)}}} >0
\text{       and      }  
\liminf_{n\to \infty} \frac{-\gamma(x_{-n}\ldots x_{-1})}{ n^{\frac{\log(\theta_2)}{\log (\theta_1)}}} >0
$$
\end{lemma}
\begin{proof}
We only prove the first inequality, the other one can be shown analogously.
Assume the result does not hold. Then, for a subsequence $(n_i)_{i\in \NN}$, 
$$\lim_{i \to \infty} \frac{\gamma(x_0\ldots x_{n_i})}{ n_i^{\frac{\log(\theta_2)}{\log (\theta_1)}}} =0$$

Let $(p_i,c_i,s_i)_{i\in \NN}$ be the prefix-suffix decomposition of $x$ and let $a \in \hat A$ such that $(p^{(a)},a,s^{(a)})$ is the periodic part of it. 

(1) First we assume $s^{(a)}$ is different from the empty word. Let $N_i$ be the minimal integer such that $x_1\ldots x_{n_i}$ is the prefix of $\sigma^{N_i}(a)$. 

Consider the prefix-suffix decomposition $(p^{(n_i)}_j,c^{(n_i)}_j,s^{(n_i)}_j)_{j\in \NN}$ of $T^{n_i+1}(x)$. Clearly, 

$$\sigma^{N_i-1}(p^{(n_i)}_{N_i-1})\ldots \sigma(p^{(n_i)}_1) p^{(n_i)}_0=
\sigma^{N_i-1}(p_{N_i-1})\ldots \sigma(p_1) p_0 x_0\ldots x_{n_i}$$

Then, 
$$\sum_{j=N_i-1}^{0} \theta_2^{j} \gamma(p^{(n_i)}_j)=
\sum_{j=N_i-1}^{0} \theta_2^{j} \gamma(p_j) + \gamma(x_0\ldots x_{n_i})$$

Dividing by $\theta_2^{N_i}$ one gets, 
$$\sum_{j=1}^{N_i} \theta_2^{-j} \gamma(p^{(n_i)}_{N_i-j})=
\sum_{j=1}^{N_i} \theta_2^{-j} \gamma(p_{N_i-j}) +\theta_2^{-N_i}  \gamma(x_0\ldots x_{n_i}) $$

Taking the limit when $i \to \infty$ and using the fact that $x$ is minimal one gets
$$\lim_{i\to \infty} \sum_{j=1}^{N_i} \theta_2^{-j} \gamma(p^{(n_i)}_{N_i-j})=v(a)$$
since by assumption $\lim_{i \to \infty} \theta_2^{-N_i}  \gamma(x_0\ldots x_{n_i})=0$. Observe that 
$n_i$ behaves like $\theta_1^{N_i}$.

This property allows to show, following the same ideas used to prove Lemma \ref{lem:series},
that there is $\by=(p^\by_i,c^\by_i,s^\by_i)_{i\in \NN} \in \SFTD$ such that 
$v(a;\by)=v(a)$. By Lemma \ref{lem:unico}, $\by_1=(p^{(a)},a,s^{(a)})$. This implies $n_i+1=0$ for some large $i$, which is a contradiction. 

(2) Now suppose $s^{(a)}$ is the empty word. Then, (considering a power of $\sigma$ if necessary)  
$(x_n)_{n\geq N}=\lim_{m\to \infty} \sigma^m(b)$ for  some $N\in \NN$ and $b\in A$. If we write 
$\sigma(b)=bs$ one obtains $x_m \ldots= b s \sigma(s) \sigma^2(s) \ldots$. 

We claim $v(b)=0$. Suppose this is not true. Then for $k \in \NN$ large enough one has
$\sum_{i=1}^k \theta_2^{K-i} \gamma(p^{\bx(b)}_i)\leq  K \theta_2^k$ with $K<0$.  That is, 
$\gamma$ applied to a prefix of $\sigma^k(b)$ can be as negative as we want if $k$ increases. 
This implies that $\gamma_n(x) <0$ for some $n \in \NN$, which is imposible since $x$ is a minimal point. Then $v(b)=0$. Furthermore, we have proved that $\gamma(x_N\ldots x_{N+i})>0$ for all $i\geq 1$. One also deduces, by the algebraic condition, that $\bx(b)=(\varepsilon,b,s)_{i\in \NN}$, where $\varepsilon$ is the empty word. 

To conclude one uses part (1) with $b$ instead of $a$. 

\end{proof}

The following proposition is plain.

\begin{prop}\label{prop:expseries}
Let $x \in \xs$ be a minimal point. Then,
$$\sum_{n\geq 1} e^{{-\gamma(x_0\ldots x_{n-1})}} < \infty  \text{ and } \sum_{n\geq 1} 
e^{{\gamma(x_{-n}\ldots x_{-1})}} <\infty$$
\end{prop} 

%%%%%%%%%%%%%%%%%%%%%%%%%%%%%%%%%%%%%%%%%%%%%%%%%%%%%%%%%%%%%%%%%%%%%%%%%%%%%%%%%
\section{Proof of the Main Theorem} \label{proof-main}

%%%%%%%%%%%%%%%%%%%%%%%%%%%%%%%%%%%%%%%%%%%%%%%%%%%%%%%%%%%%%%%%%%%%%%%%%%%%%%%%%%

The arguments of this section follows the strategy developed in  the works of \cite{Ca} and \cite{Co}. 

Let $\IET$ be a self-similar interval exchange transformation and $R$ its associated matrix. 
Assume $R$ verifies hypotheses of Theorem \ref{main}. 

Let $\xs$ be the substitutive system associated to $\IET$ and let $M={^tR}$ be the associated matrix. Consider a minimal point $x \in \xs$. 
By Proposition \ref{prop:expseries}, 
$$K=  \sum_{n\geq 1} e^{{\gamma(x_{-n}\ldots x_{-1})}} + 1+ \sum_{n\geq 1} e^{{-\gamma(x_0\ldots x_{n-1})}} <\infty$$

Let $t=\varphi(x)$. That is, $x$ is the coding of $t$ or $x$ is the coding of 
$(\lim_{s \to t^-} T^i(s))_{i\in \ZZ}$ in the case $t$ is in the orbit of one of the $a_i$'s.  
To simplify notations we assume the first case holds, the other one is analogous.

Define the probability measure $\mu_t$ on $[0,1)$ by 
$$\mu_t= \frac{1}{K} \left (  \sum_{n\geq 1} 
e^{{\gamma (x_{-n} \ldots x_{-1}) }} \delta_{\IET^{-n}t} + \delta_t + \sum_{n\geq 1} 
e^{{-\gamma(x_0\ldots x_{n-1})}} \delta_{\IET^{n}t} \right ) $$

\begin{lemma}\label{lem:goodmu} For every Borel set $I \subseteq [0,1)$
$$\mu_t(\IET(I))= \sum_{i=1}^r e^{-\gamma_i} \mu_t(I\cap [a_{i-1},a_i))$$ 
\end{lemma}
\begin{proof}
It is enough to consider $I= [a_{i-1},a_i)$ for $i\in A$.  One has, 

\begin{align*}
&\mu_t(\IET(I)) \\
&= \frac{1}{K} \left (  \sum_{n\geq 1} 
e^{{\gamma (x_{-n} \ldots x_{-1}) }} \delta_{\IET^{-n}t}+ \delta_t +   
 \sum_{n\geq 1} e^{{-\gamma(x_0\ldots x_{n-1})}} \delta_{\IET^{n}t} \right )(\IET(I))  \\
&= \frac{1}{K} \left (  \sum_{n\geq 1} 
e^{{\gamma (x_{-n} \ldots x_{-1}) }} \delta_{\IET^{-n-1}t} + \delta_{\IET^{-1}t}+ 
\sum_{n\geq 1} e^{{-\gamma(x_0\ldots x_{n-1})}} \delta_{\IET^{n-1}t}\right ) (I) \\
&= \frac{1}{K} \left ( \sum_{n\geq 1} 
e^{-\gamma(x_{-n})}  e^{{\gamma (x_{-n} \ldots x_{-1}) }} \delta_{\IET^{-n}t}  
 + e^{-\gamma(x_{0})} 
 \delta_{t}+
\sum_{n\geq 1} e^{-\gamma(x_n)} e^{{-\gamma(x_0\ldots x_{n-1})}} \delta_{\IET^{n}t} \right ) (I) \\
&= e^{-\gamma_i} \mu_t(I)
\end{align*}
where in the last equality we use the fact that $\IET^n(t) \in I$ if and only if $\gamma(x_n)=\gamma_i$.
\end{proof}

Define  $g: [0,1) \to [0,1)$ by $g(s)=\mu_t([0,s])$. This function is nondecreasing, right continuous and has left limits. Let $i\in A$. Denote $a'_i=T(a_i)$ and define $b_i=\lim_{a \to a_i^-}g(a)$ and 
$b'_i=\lim_{a' \to (a'_i)^-}g(a')$. Then at interval $[b_{i-1},b_i)$ define  linearly the AIET $f$ with image 
$[b_{i-1}',b_i')$. The slope vector of $f$ is $w=(e^{-\gamma_1},\ldots,e^{-\gamma_r})$. 
Indeed, 
$$\frac{b_i'-b'_{i-1}}{b_i-b_{i-1}}=\frac{\mu_t([a'_{i-1},a'_i))}{\mu_t([a_{i-1},a_i))}=e^{-\gamma_i}$$ 
where the last equality follows from Lemma \ref{lem:goodmu}.

Let $h: [0,1) \to [0,1)$ be the map defined by:
$h(v)=u$ if  $g(u)=v$ and $h(v)=u$ if $\lim_{w \to u^-} g(w)\leq v \leq g(u)$. Clearly $h$ is surjective, continuous and non decreasing.  Since $\mu_t$ has atoms, then $h$ is not injective  

The following lemma allows to conclude Theorem \ref{main}.

\begin{lemma}\label{lem:final}
The map $h$ defines a semi-conjugacy between the AIET $f$ and $\IET$. Moreover, $f$ has 
wandering intervals.
\end{lemma}
\begin{proof}
The semi-conjugacy follows from construction. The interval $$I=(\lim_{s\to t^-}g(s),g(t)]$$ is a wandering interval for $h$.
\end{proof}

\section{Pseudo-Anosov diffeomorphisms and eigenvalues of matrices obtained by Rauzy induction} \label{pseudoAnosov}

In this section, we discuss the hypothesis of Theorem \ref{main} in a geometric language. 
Our hypothesis is that the Perron-Frobenius eigenvalue $\theta_1$ of the matrix $R$ has a real conjugate $\theta_2 > 1$. 

We recall that every interval exchange transformation $T_{(\lambda, \pi)}$ is  realized as the first return map of a flow on a translation surface $\S$ which genus $g(\pi)$ only depends on the permutation $\pi$ (and not on $\lambda$). This translation surface is not unique. 
If $T_{(\lambda, \pi)}$ is a periodic point of the Rauzy induction, one can choose $\S$ fixed by a pseudo-Anosov diffeomorphism $\phi$ (see \cite{Th} for an enlightening discussion on pseudo-Anosov diffeomorphisms). The eigenvalue $\theta_1$ is the dominant eigenvalue of the action of $\phi$ on the absolute homology of $\S$.
Therefore  $\theta_1$ is an algebraic number of degree at most $2g(\pi)$ over $\QQ$. 

Heuristically, after the work of Avila and Viana \cite{AV}, it is reasonable to believe that a ``generic" pseudo-Anosov satisfies our hypothesis.
Nevertheless, it seems extremely difficult to understand the eigenvalues of  {\it all} pseudo-Anosov diffeomorphisms. In this section, we want to explain that our hypothesis are often satisfied. 
They are not always satisfied: for instance, the conjugates of the Arnoux-Yoccoz pseudo-Anosov are not real. Situations much worse do exist.

\subsection{Existence of a conjugate $\theta_2$ with $|\theta_2| \geq 1$}

A pseudo-Anosov diffeomorphism preserves the symplectic form induced by the intersection form. Thus if $z$ is an  eigenvalue of the automorphism $\phi_*$ of $H_1(\S, \ZZ)$, its inverse $z^{-1}$ is also an eigenvalue of $\phi_*$. Consequently, $\frac{1}{\theta_1}$ is an eigenvalue of $\phi_*$. If it is the only Galois conjugate of $\theta_1$, it means that $\theta_1$ is an algebraic number of degree 2.
It is classical (see \cite{KS} for instance) that the surface $\S$ is then a covering of a torus (a square tiled surface up to normalization). 
Therefore hypothesis (\ref{algebraic-hyp}) is satisfied if and only if the surface $\S$ is not a square tiled surface. Thus, this hypothesis is very natural and simple to check. 
 
\subsection{Real conjugates} 
The second hypothesis is more subtle to analyze. 

A pseudo-Anosov diffeomorphism is obtained by Thurston's construction if it is the product of two affine Dehn twists $T_h$ and $T_v$ along two multi-curves filling a surface (see \cite{Th}).

After normalization,  the derivatives of the Dehn twists in the natural parameters of the translation surface are 
$$T_h = \begin{pmatrix}
1 & a \\
0 & 1
\end{pmatrix}, \
T_v = \begin{pmatrix}
1 & 0 \\
b & 1
\end{pmatrix}$$
\noindent where $a$ and $b$ are positive real numbers and $ab$ is an algebraic number.

An element $f$ of the group generated by $T_h$ and $T_v$ is a pseudo-Anosov diffeomorphism if the absolute value of the trace $t(f)$ of the corresponding matrix is larger than 2.
For every pseudo-Anosov diffeomorphism obtained by Thurston's construction,  the conjugates of $ t(f)$  are real numbers (see \cite{HL}). The dominant root of the action  of $f$ on the homology is the real number $\theta_1 > 1$ with $\theta_1 + \theta_1^{-1} = t(f)$.
The number $\theta_1$ (or one of its power) is the Perron-Frobenius eigenvalue of the matrix obtained by Rauzy induction considered in the present paper (see \cite{Ve}). 
Let $\theta'$ be a conjugate of $\theta_1$ and $t'(f) = \theta' + \theta'^{-1}$  a ({\it real}) conjugate of $t(f)$.
$\theta'$ is a real number with $\theta' > 1$ if $\vert t'(f) \vert > 2$.  It is a complex number of modulus one if $\vert t'(f) \vert < 2$.
This directly comes from the fact that  $\theta' + \theta'^{-1} = t'(f)$.

For instance, the diffeomorphisms $f_{n,m} = T_h^nT_v^m$ are pseudo-Anosov diffeomorphisms if $n, m$ are positive integers. In fact  the absolute value of the trace of 
$\begin{pmatrix}
1 & a \\
0 & 1
\end{pmatrix}^n
\begin{pmatrix}
1 & 0 \\
b & 1
\end{pmatrix}^m
$ is larger than 2 because $nm ab >0$.

Thus $\theta' > 1$ if $\vert t'(f) \vert = \vert 2+ nm (ab)' \vert > 2$ (where $(ab)'$ is a real number). This is satisfied for all couples $(n,m)$ except for finite number of exceptions. Using more sophisticated argument, $\vert t'(f) \vert > 2$ if $n$ and $m$ are positive integers.

\bigskip

{\bf Acknowledgments.} The second author is supported by project blanc ANR: ANR-06-BLAN-0038. The third author is supported by Nucleus
Millennium Information and Randomness P04-069-F.

\end{document}